\documentclass[11pt]{article}
\usepackage{amsmath}
\usepackage{amsthm}
\usepackage{amssymb}
\usepackage{fullpage}
\usepackage{authblk}
\usepackage{graphics}
\usepackage{bm}
\usepackage{setspace}
\usepackage[numbers,sort&compress]{natbib}

\usepackage{amsfonts}
\usepackage{eucal}
\usepackage{latexsym}
\usepackage{mathdots}
\usepackage{mathrsfs}
\usepackage{enumitem}

\usepackage[all]{xy}

\usepackage{color}

\newcommand{\be}{\begin{equation}}
\newcommand{\ee}{\end{equation}}
\newcommand{\ba}{\begin{eqnarray}}
\newcommand{\ea}{\end{eqnarray}}
\newcommand{\bal}{\begin{align}}
\newcommand{\eal}{\end{align}}
\newcommand{\baln}{\begin{align*}}
\newcommand{\ealn}{\end{align*}}
\newcommand{\bi}{\begin{itemize}}
\newcommand{\ei}{\end{itemize}}
\newcommand{\bn}{\begin{enumerate}}
\newcommand{\en}{\end{enumerate}}
\newcommand{\bbm}{\begin{bmatrix}}
\newcommand{\ebm}{\end{bmatrix}}
\newcommand{\bpm}{\begin{pmatrix}}
\newcommand{\epm}{\end{pmatrix}}
\newcommand{\bsm}{\left ( \begin{smallmatrix}}
\newcommand{\esm}{\end{smallmatrix} \right) }
\newcommand{\bp}{\begin{proof}}
\newcommand{\ep}{\end{proof}}
\newcommand{\nn}{\nonumber}

\newcommand{\mr}{\ensuremath{\mathrm}}
\newcommand{\scr}{\ensuremath{\mathscr}}
\newcommand{\mbf}{\ensuremath{\mathbf}}
\newcommand{\mc}{\ensuremath{\mathcal}}
\newcommand{\mf}{\ensuremath{\mathfrak}}

\newcommand{\wt}{\ensuremath{\widetilde}}

\newcommand{\Ga}{\ensuremath{\Gamma}}
\newcommand{\ga}{\ensuremath{\gamma}}
\newcommand{\Om}{\ensuremath{\Omega}}

\newcommand{\la}{\ensuremath{\lambda }}

\newcommand{\eps}{\ensuremath{\epsilon }}

\def\C{\mathbb{C}}

\def\D{\mathbb{D}}

\def\N{\mathbb{N}}
\def\B{\mathbb{B}}

\def\cH{\mathcal{H}}
\def\cG{\mathcal{G}}
\def\bH{\mathbb{H}}
\def\cJ{\mathcal{J}}

\def\hardy{\mathbb{H} ^2 _d}
\def\mult{\mathbb{H} ^\infty _d}

\def\bc{\underset{\mbox{\tiny \mbox{\tiny b.c.}}}{\subseteq}}
\def\cc{\underset{\mbox{\tiny \mbox{\tiny c.c.}}}{\subseteq}}

\renewcommand{\H}{\ensuremath{\mathcal{H} }}

\newcommand{\J}{\ensuremath{\mathcal{J} }}

\newcommand{\M}{\ensuremath{\mathfrak{M} }}
\newcommand{\K}{\ensuremath{\mathcal{K} }}
\renewcommand{\L}{\ensuremath{\mathscr{L} }}
\newcommand{\F}{\ensuremath{\mathbb{F} }}

\newcommand{\cF}{\ensuremath{\mathcal{F}}}

\newcommand{\ip}[2]{\ensuremath{\langle {#1} , {#2} \rangle}}
\newcommand{\ipcn}[2]{\ensuremath{\left( {#1} , {#2} \right) _{\C ^n}}}

\renewcommand{\dim}[1]{\ensuremath{\mathrm{dim} \left( {#1} \right) }}
\newcommand{\ran}[1]{\ensuremath{\mathrm{Ran} \left( {#1} \right) }}

\renewcommand{\ker}[1]{\ensuremath{\mathrm{Ker} \left( {#1} \right) }}

\newcommand{\re}[1]{\ensuremath{\mathrm{Re} \left( {#1} \right) }}

\newcommand{\hilbmod}{\ensuremath{\mathrm{Hilb}}}
\newcommand{\cC}{\ensuremath{\mathcal{C}}}
\newcommand{\cS}{\ensuremath{\mathcal{S}}}
\newcommand{\sM}{\ensuremath{\mathscr{M}}}
\newcommand{\sN}{\ensuremath{\mathscr{N}}}

\numberwithin{equation}{section}
\numberwithin{subsection}{section}

\newtheorem{thm}[subsection]{Theorem}

\newtheorem{lemma}[subsection]{Lemma}
\newtheorem{prop}[subsection]{Proposition}
\newtheorem{cor}[subsection]{Corollary}
\newtheorem*{thm*}{Theorem}

\theoremstyle{definition}
\newtheorem{defn}[subsection]{Definition}
\newtheorem{remark}[subsection]{Remark}

\title{A de Branges-Beurling theorem for the full Fock space}

\author[1]{Robert T.W. Martin}
\affil[1]{\footnotesize University of Manitoba}

\author[2]{Eli Shamovich}
\affil[2]{ \footnotesize Ben-Gurion University of the Negev}

\date{}

\begin{document}
\bibliographystyle{abbrvnat}
\maketitle

\begin{abstract}
We extend the de Branges-Beurling theorem characterizing the shift-invariant spaces boundedly contained in the Hardy space of square-summable power series to the full Fock space over $\C ^d$. Here, the full Fock space is identified as the \emph{Non-commutative (NC) Hardy Space} of square-summable Taylor series in several non-commuting variables. We then proceed to study lattice operations on NC kernels and operator-valued multipliers between vector-valued Fock spaces. In particular, we demonstrate that the operator-valued Fock space multipliers with common coefficient range space form a bounded general lattice modulo a natural equivalence relation.
\end{abstract}
\onehalfspace

\section{Introduction}

The classical Hardy space, $H^2$, of the complex unit disk, $\D = (\C ) _1$, is the Hilbert space of all analytic functions in the disk with square-summable Taylor series coefficients at the origin (equipped with the $\ell ^2$ inner product of these coefficients). The Hardy algebra, $H^\infty$, is the unital Banach algebra of all uniformly bounded analytic functions in $\D$ with the supremum norm, and this can be identified with the \emph{multiplier algebra} of $H^2$, the algebra of all functions in $\D$ which multiply $H^2$ into itself.  That is, if $h \in H^\infty$ and $f \in H^2$, then $h \cdot f \in H^2$. In this way any element $h \in H^\infty$, or \emph{multiplier}, defines a bounded multiplication operator $M_h :H^2 \rightarrow H^2$, $M_h f := hf$, and the operator norm of $M_h$ is equal to the supremum norm of $h$. A multiplier $h \in H^\infty$ is called \emph{inner} if the multiplication operator $M_h$ is an isometry on $H^2$. In particular $h(z) = z$ is an inner function, and the isometry $S:= M_z$, the \emph{shift}, plays a central role in operator theory on Hardy spaces \cite{Nik-shift,NF}. Beurling's theorem, for example, identifies the closed shift-invariant subspaces of $H^2$ as the ranges of inner functions - this is a celebrated and fundamental result in the classical theory \cite{Beu48,Hoff}.  This result was further extended in a natural way by de Branges: the de Branges - Beurling theorem identifies any shift-invariant space of power series boundedly contained in $H^2$ as the operator-range space of a bounded multiplier \cite{Sarason-dBBthm}. 

A canonical multi-variable extension of the Hardy space $H^2$ is then the full Fock space over $\C ^d$, $\hardy$.  The Fock space can be defined as the space of all power series in $d$ non-commuting formal variables with square-summable coefficients. That is, any $f \in \bH ^2 _d$ is a formal power series:
$$ f = f (\mf{z} ) = \sum _{\alpha \in \F ^d} \hat{f} _\alpha \mf{z} ^\alpha, $$ where $\mf{z} _1 , \cdots , \mf{z} _d$ are $d-$non-commuting (NC) formal variables and $\F ^d$ is the set of all words in the $d$ letters $\{ 1, \cdots , d \}$ (including the empty word $\emptyset$ containing no letters). For example, if $\alpha = 122112$, then the NC monomial $\mf{z} ^\alpha$ is $ \mf{z}_1 \mf{z} _2 ^2 \mf{z} _1 ^2 \mf{z} _2$. In this multi-variable setting there are natural $d-$tuples of left and right \emph{NC shifts}. That is for each of the NC variables $\mf{z} _k$, the left multiplication operator $L_k := M ^L _{\mf{z} _k}$ is an isometry on $\bH ^2 _d$, and these have pairwise orthogonal ranges, 
$$ L_k ^* L_j = \delta _{k,j} I, $$ so that the $d-$tuple $L := \left( L _1 , \cdots , L_d \right) : \bH ^2 _d \otimes \C ^d \rightarrow \bH ^2 _d$ defines a \emph{row isometry}, \emph{i.e.} an isometry from several copies of a Hilbert space into itself. We call this isometry the \emph{left free shift}.  There is an exact version of the Wold decomposition in this multi-variable setting due to Popescu and the left free shift is the universal (pure) row isometry \cite{Pop-dil}. Similarly, one can also define isometric right free shifts, $R_k := M^R _{\mf{z} _k}$, and these are in fact unitarily equivalent to the left free shifts under a self-adjoint transpose unitary, $U_\dag$ on $\bH ^2 _d$. NC $H^\infty$ can then be defined as $\bH ^\infty _d := \mr{Alg} (I , L) ^{-weak-*}$, and this can be viewed as the \emph{left NC multiplier algebra} of $\bH ^2 _d$. Similarly $R^\infty _d = \mr{Alg} (I, R) ^{-weak-*}$ is the \emph{right NC multiplier algebra}. Popescu \cite{Pop-charfun,Pop-multfact} (and later Davidson-Pitts \cite{DP-inv}) have obtained an exact NC analogue of the Beurling theorem for $\bH ^2 _d$: 

\begin{thm*}{ NC Beurling Theorem (Popescu/ Davidson-Pitts)}
Any $R-$cyclic, $R-$invariant subspace of $\bH ^2 _d$ is the range of a left inner (isometric) multiplier $\Theta (L) \in \bH ^\infty _d$. Any $R-$invariant subspace of $\bH ^2 _d$ is the direct sum of $R-$invariant, $R-$cyclic subspaces.
\end{thm*}

In this paper we study operator-valued left multipliers between vector-valued NC Hardy spaces, $\bH ^2 _d \otimes \cH$ and $\bH ^2 _d \otimes \cJ$. We will show that the set of all operator valued left multipliers, $\bH ^\infty _d \otimes \mc{L} (\cdot , \cJ)$ with range contained in $\bH ^2 _d \otimes \cJ$ can be equipped with natural lattice operations $\wedge, \vee$ so that given any $F \in \bH ^\infty _d \otimes \mc{L} (\mc{F} , \mc{J} ) $ and $G \in \bH ^\infty _d \otimes \mc{L} (\mc{G} , \mc{J} )$, both $F \vee G$ and $F \wedge G$ are operator-valued left multipliers with range contained in $\bH ^2 _d \otimes \cJ$. We further develop factorizations of $F \wedge G$ in terms of $F \vee G$ and $F \oplus G$. The operator-valued left multiplier $F \wedge G$, in particular, is defined via an operator-valued NC version of the de Branges-Beurling theorem:

\begin{thm*}[NC de Branges-Beurling Theorem (Theorem \ref{NCdBB})]
A Hilbert space $\scr{M}$ is boundedly contained in $\bH ^2 _d \otimes \cJ$ and $R-$invariant if and only if it is the operator-range space of an operator valued left multiplier $F (L) \in \bH ^\infty _d \otimes \mc{L} (\mc{F} , \mc{J} )$.
\end{thm*}

As a consequence of this theorem we obtain a modest generalization of a result of Davidson and Pitts on right ideals in $\bH^{\infty}_d$. In Sections \ref{sec:lattice_kernels} and \ref{sec:lattice_mult} we discuss lattice operations on kernels and multipliers. The classical analog of these results are the greatest common divisor (gcd) and least common multiple (lcm) of inner functions. We show that modulo a certain natural equivalence relation, the multipliers form a bounded general lattice with operations of join and meet. The last section provides a different viewpoint on this lattice. Namely, the final section describes these lattice operations in terms of the category of right Hilbert modules over the free algebra.

\section{Preliminaries: NC function theory and NC reproducing kernels}

The Hardy space $H^2$ of the disk can be equivalently defined as the \emph{reproducing kernel Hilbert space} (RKHS) of the positive \emph{Szeg\"{o} kernel}, $k : \D \times \D \rightarrow \C$: 
$$ k(z,w) := \frac{1}{1-zw^*}; \quad \quad z,w \in \D, $$ $H^2 (\D) = \cH (k)$. Here, recall that a RKHS, $\cH$, is any Hilbert space of functions on some set $X$ (say, a Hausdorff topological space) so that point evaluation at any $x \in X$ is a bounded linear functional. By the Riesz Lemma there is then a point evaluation vector, or \emph{kernel} at $x$, $k_x \in \cH$ so that 
$ \ip{k_x}{f} = f(x), $ for any $f \in \cH$. The two-variable function $k(x,y) := \ip{k_x}{k_y}_\cH$ is then a \emph{positive kernel function on $X$} in the sense that for any finite subset of $X$, the Gram matrix $[ k(x_i , x_j ) ] \geq 0$ is positive semi-definite. One then writes $\cH =\cH (k)$, and the classical theory of reproducing kernels due to Aronszajn and Moore shows there is a natural bijection between positive kernels on a given set $X$, and RKHS of functions on that set.

In this paper, the operator-range spaces of NC left or right multipliers can be viewed as \emph{non-commutative reproducing kernel Hilbert spaces} (NC-RKHS) of \emph{free non-commutative functions} defined on \emph{non-commutative sets}. Our use of the theory of NC functions and NC-RKHS is mostly superficial, and so we will provide a breviloquent introduction to these concepts and leave details to expert references.

Most NC-RKHS in this paper will be Hilbert spaces of NC functions defined in the \emph{NC open unit row-ball}:
$$ \B ^d _\N := \bigsqcup _{n=1} ^\infty \B ^d _n; \quad \B ^d _n := \left( \C ^{n \times n} \otimes \C ^{1\times d} \right) _1. $$ That is, each level $\B ^d _n$ is the set of all \emph{strict row contractions on $\C ^n$}, any $Z \in \B ^d _n$ is a $d-$tuple $Z = \left( Z_1 , \cdots , Z_d \right)$ of $n\times n$ matrices, $Z_k \in \C ^{n\times n}$ so that $Z : \C ^{n} \otimes \C ^d \rightarrow \C ^n$ defines a strict contraction from $d$ copies of $\C ^n$ into one copy:
$$ Z Z ^* = Z_1 Z_1 ^* + \cdots Z_d Z_d ^* < I_n. $$ 
In brief, a \emph{free NC function} on $\B ^d _\N$ is any function $f : \B ^d _\N \rightarrow \bigsqcup \C ^{n\times n}$ which respects the grading, joint similarities and direct sums (it is easy to check any NC polynomial has these properties). Locally (level-wise) bounded free NC functions in $\B ^d _\N$ are automatically analytic \cite[Chapter 7]{KVV}. Generally speaking, an NC set, $\Om$, is any subset of the \emph{NC universe}, $\C ^d _\N := \bigsqcup _{n=1} ^\infty \C ^{n\times n} \otimes \C ^{1 \times d}$ which is closed under direct sums. One writes $\Om := \bigsqcup \Om _n$ where $\Om _n := \Om \cap \C ^{n\times n} \otimes \C ^{1\times d}$. An NC-RKHS on $\Om$ is any Hilbert space, $\cH$ of free NC functions in $\Om$ so that point evaluation at any $Z \in \Om _n$, $\ell _Z$, is a bounded linear map from $\cH$ into the Hilbert space $\left( \C ^{n\times n} , \mr{tr} _n \right)$. The Hilbert space adjoint of this linear map is the kernel map at $Z$,  
$$ K_Z := \ell _Z ^* : \C ^{n\times n} \rightarrow \cH, $$ and the \emph{completely positive non-commutative (CPNC) kernel} of $\cH$ is defined as follows: Given any $Z \in \Om _n , \ W \in \Om _m$, $K(Z,W) [ \cdot ] : \C ^{n\times m} \rightarrow \C ^{n \times m}$ is given by:
$$ \ipcn{y}{K(Z,W) [vu^*] x} := \ip{K_Z (yv^*)}{K_W (xu^*)}_\cH; \quad y, v \in \C ^n, \ x,u \in \C ^m. $$ (We will typically write $K\{Z, y ,v \} := K _Z (y v^* )$.) For any $Z \in \Om _n$, $K(Z,Z) [ \cdot ]$ is a completely positive map. The theory of NC-RKHS is a faithful analogue of the classical RKHS theory and there is again a bijection between CPNC kernels and NC-RKHS of free NC functions on NC sets \cite{BMV}. In particular the full Fock space, $\bH ^2 _d$ can be identified with the \emph{NC Hardy space}, the NC-RKHS of all free NC functions in $\Om = \B ^d _\N$ corresponding to the \emph{NC Szeg\"{o} kernel}:
$$ K(Z,W) := \sum _{\alpha \in \F ^d} Z^\alpha [ \cdot ] (W^* ) ^\alpha. $$ As mentioned in the introduction, $\bH ^\infty _d := \mr{Alg} (I , L) ^{-WOT}$ can then be identified with the \emph{left multiplier algebra} of $\bH ^2 _d$. Namely, given any $h \in \bH ^2 _d$ and $H \in \bH ^\infty _d$, $H \cdot h \in \bH ^2 _d$, and the linear multiplication operator, 
$H(L) := M^L _H : \bH ^2 _d \rightarrow \bH ^2 _d$, $( H(L) h) (Z) = H (Z) h (Z)$ is bounded with operator norm equal to the supremum norm of its symbol, $H$, over the NC unit ball \cite[Theorem 3.1]{SSS}. 

\subsection{Row contractions and their row isometric dilations}

\label{minimal}

Let $X$ be a row contraction on a Hilbert space, $\mc{H}$.  We say that a row isometry, $\hat{X}$ on $\hat{\mc{H}} \supseteq \mc{H}$ is a row isometric dilation of $X$, if compression to $\mc{H}$ is a unital homomorphism of $\mr{Alg} (I , \hat{X} )$ onto $\mr{Alg} (I, X)$:
$$ P_\mc{H} \hat{X} ^\alpha P_{\mc{H}} = X ^\alpha P _{\mc{H}}; \quad \quad \alpha \in \F ^d. $$ 
Such a dilation, $\hat{X}$, is called \emph{minimal} if $\mc{H}$ is $\hat{X}-$cyclic, \emph{i.e.} 
$$ \hat{\mc{H}} = \bigvee _{\alpha \in \F ^d} \hat{X} ^\alpha \mc{H}. $$  Given any row isometric dilation $(\hat{X}, \hat{\mc{H}} )$ of $X$, observe that if one defines
$$ \hat{\mc{H}}_0 := \bigvee _{\alpha \in \F ^d} \hat{X} ^\alpha \mc{H}, \quad \mbox{and} \quad 
\hat{X} _0 := \hat{X} | _{\hat{\mc{H}} _0 } \otimes \C ^d, $$ that $( \hat{X} _0, \hat{\mc{H}} _0 )$ is a minimal row isometric dilation of $X$. Two row isometric dilations of $X$, $(\hat{X}, \hat{\mc{H}})$, $(\hat{X} ' , \hat{\mc{H}} ' )$ are said to be \emph{equivalent}, if there is an onto isometry $U : \hat{\mc{H}} \rightarrow \hat{\mc{H}} '$ so that $Uh =h$ for all $h \in \H \subseteq \hat{\mc{H}}, \hat{\mc{H}} ' $, and 
$$ U \hat{X} ^\alpha = (\hat{X} ') ^\alpha U. $$ 

As proven in \cite[Theorem 2.1]{Pop-dil}, any row contraction, $X$ on $\mc{H}$ has a minimal row isometric dilation $(\hat{X}, \hat{\mc{H}} )$ which obeys the property that $\mc{H}$ is $\hat{X}-$coinvariant, and,
$$ \hat{X} ^* | _{\mc{H}} = X ^*. $$ Moreover, any minimal row isometric dilation of $X$ with these properties is unique up to the above notion of equivalence. 

\begin{lemma} \label{minunique}
Let $X$ be a row contraction on the Hilbert space $\mc{H}$. Then any two minimal row isometric dilations $(\hat{X} , \hat{\mc{H}})$ and $(\hat{X} ' , \hat{\mc{H}} ' )$ of $X$ are equivalent and $\mc{H}$ is co-invariant for any minimal row isometric dilation of $X$.
\end{lemma}
This is easily established as in the proof of \cite[Chapter 4,Theorem 4.1]{NF}. Namely, given any two minimal isometric dilations, $\hat{X}, \wt{X}$ of $X$, one verifies that the linear map
$U :\hat{\mc{H}} \rightarrow \hat{\mc{H}} '$ defined by
$$ U \hat{X} ^\alpha h = (\hat{X} ') ^\alpha h; \quad \quad h \in \mc{H}, \ \alpha \in \F^d, $$ defines an onto isometry between the Hilbert spaces of the two minimal dilations.

A row contraction, $T = (T _1 , \cdots , T_d ) : \cH \otimes \C ^d \rightarrow \cH$ is said to be \emph{pure} (or of class $C _{\cdot \ 0}$) if:
$$ \lim _{n \rightarrow \infty } \sum _{|\alpha | = n} \| T^{*\alpha} h \| = 0. $$ 
By \cite[Proposition 2.3]{Pop-dil}, $T$ is pure if and only if its minimal row isometric dilation is unitarily equivalent to copies of $L$, or equivalently to copies of $R$. Here, note that $R_k$ and $L_k$ are unitarily equivalent via the idempotent \emph{transpose unitary} on $\bH ^2 _d$:
$$ U_\dag L^\alpha 1 = L^{\alpha ^\dag} 1, $$ where if $\alpha = i _1 \cdots i _n$, then $\alpha ^\dag = i _n \cdots i _1$. That is, $R_k = U_\dag L_k U_\dag$ for $1 \leq k \leq d$.

\section{Range containment of left multipliers}

\begin{thm}{ (NC Douglas Factorization property \cite[Theorem 5.5]{JM-freeSmirnov})}
Let $F(L) \in \bH ^\infty _d \otimes \scr{L} (\mc{F}, \J)$ and $G(L) \in \bH ^\infty _d \otimes \scr{L} (\mc{G} , \J )$ be left free multipliers so that $\ran{F(L)} \subseteq \ran{G(L)}$. There is a unique left free multiplier
$H \in \bH ^\infty _d (\mc{F} , \mc{G} )$ so that $F = G H $, $\ker{H(L)} \subseteq \ker{F(L) }$ and
$$ \| H(L) \| ^2 = \inf \{ \la ^2 \geq 0 | \ F(L)  F(L) ^* \leq \la ^2 G(L) G(L) ^* \}. $$ 
\end{thm}
Observe that if $G(L)$ is injective, then $\ker{H(L)} = \ker{F(L)}$.
\begin{proof}
By the Douglas Factorization Lemma, \cite{DFL}, since $\ran{F(L)} \subseteq \ran{G(L)}$, there is a unique $H \in \mc{L} (\bH ^2 _d \otimes \mc{F} , \bH ^2 _d \otimes \mc{G} )$, so that 
$$ F(L) = G(L) H, $$ and so that $H$ is uniquely determined by the properties that 
$$ \| H \| ^2 = \inf \{ \la ^2 \geq 0 | \ F(L)  F(L) ^* \leq \la ^2 G(L) G(L) ^* \}, $$
$\ker{H} = \ker{F(L)}$, and $\ran{H} \subseteq \ker{G(L)} ^\perp$. Note that $\ker{G(L)}$ is closed and $R\otimes I_\mc{G}-$invariant. Setting $P := P_{\ker{G}} ^\perp$, this is a $R-$co-invariant projection, and let $T := P R \otimes I_\mc{G} | _{\ker{G} ^\perp}$. This is a pure row contraction with row isometric dilation $R \otimes I_{\mc{G}}$. Similarly, let $P'$ be the $R-$coinvariant projection $P_{\ker{F} } ^\perp$, and $T' := P ' R \otimes I_{\mc{F}} | _{\ker{F} ^\perp}$. Since $F(L)$ and $G(L)$ intertwine right free shifts, observe that:
\ba R _k  \otimes I _{\mc{J}} F(L) & = & G(L) R_k \otimes I _{\mc{G} } H \nn \\
& = & F(L) R_k \otimes I _{\mc{F}}, \nn \ea and it follows that 
\ba T_k H & = & P R_k \otimes I_{\mc{G}} H \nn \\
& = & H P ' R_k \otimes I_{\mc{F}} \nn \\
& = & H T' _k. \nn \ea 
By commutant lifting \cite{Pop-dil}, it follows that there is a $H(L) \in \bH ^\infty _d \otimes \mc{L} ( \mc{F} , \mc{G} )$ with $\| H(L) \| = \| H \|$, so that $H(L) ^* | _{\ker{G} ^\perp} = H^*$. In particular $P H(L) = P H(L) P = P H P = H P$, and 
$$ F(L) = G(L) H = G(L) P H P = G(L) H(L). $$ 
\end{proof}

\section{A de Branges-Beurling theorem for Fock space}

Consider $F \in \bH ^\infty _d  \otimes \L ( \mc{F} , \mc{J} )$.  Let $\scr{M} ^L  (F) \subseteq \bH ^2 _d \otimes \J$ be the operator range space:
$$ \scr{M} ^L  (F) =\scr{M} (F(L)) := \ran{F(L)}; \quad \quad \| F(L) x \| _{F } := \| P_{\ker{F(L)}} ^\perp x \|  _{\bH ^2 _d}. $$ That is the norm of $\scr{M} ^L  (F)$ is defined so that $F(L) $ is a co-isometry from $\bH ^2 _d \otimes \mc{F}$ onto its range space, and $$ \scr{M} ^L (F) \underset{\mbox{\tiny \mbox{\tiny b.c.}}}{\subseteq} \bH ^2 _d \otimes \J. $$ Here the notations $\bc $ or $\cc $ will denote bounded/ contractive containment, respectively. The Hilbert space $\scr{M} ^L (F)$ is the NC-RKHS with the CPNC kernel:
$$ K^F (Z,W) [ \cdot ] := F(Z) K(Z,W) [\cdot ] \otimes I_{\mc{F}} F(W) ^*, $$ and $K(Z,W)$ is, as before, the NC Szeg\"o kernel of the full Fock space $\bH ^2 _d$. Observe that any left multiplier range space, $\scr{M} ^L (F)$ is $R \otimes I_\J -$ invariant.

\begin{lemma} \label{rowcontract}
Given any $F \in \bH ^\infty _d \otimes \mc{L} (\mc{F}, \J ) $, let $X_F := (R \otimes I_\J ) | _{\scr{M} ^L (F)} $. Then $X_F$ is a row contraction, and $(X_F ^\alpha ) ^*  F(L) x  = F(L) (R^\alpha \otimes I _{\mc{F}} ) ^* P _{\ker{F}} ^\perp x $
\end{lemma}
\begin{proof}
Clearly $\scr{M} ^L (F)$ is $R \otimes I_\J -$ invariant since 
$$ R \otimes I_\J F(L) x = F(L) (R \otimes I_\mc{F} ) x \in \ran{F(L)}. $$
For any $\mbf{x}$ in $(\bH ^2 _d \otimes \mc{F}) \otimes \C^d$,
\ba \| X_F F(L) \otimes I_d \mbf{x} \| ^2 _F & = & \| (F(L) R \otimes I_{\mc{F}} ) \mbf{x} \| _F ^2 \nn \\ 
& = & \| P_{\ker{F} } ^\perp (R \otimes I_{\mc{F}} ) \mbf{x} \| _{\bH ^2 _d \otimes \mc{F} } ^2 \nn \\
& = & \| P_{\ker{F} } ^\perp (R \otimes I_{\mc{F}} ) P_{\ker{F} } ^\perp \otimes I_d \mbf{x} \| _{\bH ^2 _d \otimes \mc{F} } ^2 \quad \quad \mbox{(By $R-$invariance of $\ker{F}$)} \nn \\
& \leq & \|  (R \otimes I_{\mc{F}} ) P_{\ker{F} } ^\perp \otimes I_d  \mbf{x} \| _{\bH ^2 _d \otimes \mc{F} } ^2 \nn \\
& = &\|  P_{\ker{F} } ^\perp \otimes I_d  \mbf{x} \| _{\bH ^2 _d \otimes \mc{F} \otimes \C ^d } ^2 \nn \\
& = & \|F(L) \otimes I_d  \mbf{x} \| ^2 _{\scr{M} ^L (F ) \otimes \C ^d}. \nn \ea 
This proves that $X_F$ is a row contraction. The adjoint action of $X_F$ is a straightforward calculation:
\ba \ip{F(L) y}{X_F ^\alpha F(L) x}_F & = & \ip{P_{\ker{F} ^\perp} y}{R^\alpha \otimes I_{\mc{F}} x}_{\bH ^2} \nn \\
& = & \ip{P _{\ker{F}} ^\perp (R^{\alpha }) ^* \otimes I_{\mc{F}} P_{\ker{F}} ^\perp y}{x} _{\bH ^2 } \nn \\
& = & \ip{F(L) (R^{\alpha }) ^* \otimes I_{\mc{F}} P_{\ker{F}} ^\perp y}{F(L) x} _{F }. \nn \ea
\end{proof}

\begin{thm}{ (NC de Branges-Beurling)} \label{NCdBB}
A linear subspace $\scr{M} \underset{\mbox{\tiny \mbox{\tiny b.c.}}}{\subseteq} \bH ^2 _d \otimes \J$ is boundedly contained in $\bH ^2 _d \otimes \J$, is $(R \otimes I_\J ) -$invariant, and $X:= (R \otimes I_\J ) |_\scr{M}$ is a row contraction if and only if there is a bounded operator-valued left multiplier $F \in  \bH ^\infty _d \otimes \L (\mc{F} , \J ) $ so that $\scr{M} = \scr{M} ^L (F)$ and $\| F (L) \| = \| \mr{e} \|$, where $\mr{e} : \scr{M} \hookrightarrow \bH ^2 _d \otimes \J$ is the bounded embedding.  
\end{thm}
This theorem is inspired by \cite[Theorem 3]{dB-ss}.
\begin{lemma} \label{pure}
If $\scr{M} \subseteq \bH ^2 _d \otimes \J$ is boundedly contained in vector-valued NC Hardy space, $R \otimes I_\J -$invariant, and $X:=R \otimes I_\J | _{\scr{M} \otimes \C ^d}$ is a row contraction, then it is a pure row contraction.
\end{lemma}
\begin{proof}
Recall that a row contraction, $T$, on $\H$ is \emph{pure}, or of class $C_{\cdot \ 0}$, if 
$ \lim _{n \rightarrow \infty} \sum _{|\alpha | = n } \| (T^* ) ^\alpha h \| \rightarrow 0, $ for any fixed $h \in \H$.
Equivalently, 
$$ \lim _{n \rightarrow \infty } \|(T^*) ^{[n]} h \| =0, $$ for any fixed $h \in \H$. Here we define:
$$ (T^*) ^{[n]} = \underbrace{(T^* \otimes I_d \otimes I_{n-1} ) (T^* \otimes I_d \otimes I_{n-2} ) \cdots T^*}_{\mbox{$n$ terms}}, $$ and $(T^*) ^{[1]} = T^*$. Observe that $(T^* ) ^{[n]}$ is a contraction for any $n \in \N$. We could further identify each $(T^* ) ^{[n]}$ with a contractive linear operator on $\H \otimes \ell ^2 (\N )$, \emph{e.g.} if $d=2$, and $\mbf{h} = ( h _1 , h_2 , \cdots ) ^T \in \H \otimes \ell ^2$, then
$$ (T^*) ^{[2]} \mbf{h} \simeq \bpm T_1 ^* T_1 ^* h _1 \\ T_1 ^* T_2 ^* h _1 \\ T_2 ^* T_1 ^* h _1 \\ T_2 ^* T_2 ^* h _1 \\ 0 \\ \vdots \epm.$$ Under this identification $(T^* ) ^{[n]} \stackrel{SOT}{\rightarrow} 0$ if and only if $T$ is a pure row contraction.

Let $\mr{e} : \scr{M} \hookrightarrow \bH ^2 _d \otimes \J$ be the bounded (and injective) embedding. Taking adjoints of 
the equation $$\mr{e} R ^\alpha \otimes I_\J | _\scr{M} = R^\alpha \otimes I_\J \mr{e}, $$ yields
$$ (X^\alpha ) ^* \mr{e} ^* = \mr{e} ^* (R^\alpha ) ^* \otimes I_\J. $$ 

It follows that for any $h \in \bH ^2 _d \otimes \J $,
\ba \| (X^*) ^{[n]} \mr{e} ^* h \| ^2 & = & \| (\mr{e} ^* \otimes I_d \otimes I_n)  (R^* \otimes I_\J ) ^{[n]} h \| ^2 \nn \\
& \leq & \| \mr{e} \| \| (R^* \otimes I_\J) ^{[n]} h \| ^2 \rightarrow 0, \nn \ea 
since $R \otimes I_\J$ is pure. More generally, note that since $\mr{e}$ is injective, $\mr{e} ^*$ has dense range in $\scr{M}$. Hence, given any $g \in \scr{M}$, we can find a sequence $(h_k) \subset \bH ^2 _d \otimes \J$ so that $g_k := \mr{e} ^* h_k \rightarrow g$ in the norm of $\scr{M}$. Given any $\eps >0$, choose $K \in \N$ so that $k> K $ implies that $\| g_k - g \| _{\scr{M}} < \eps$. Then, 
$$ \| (X^*) ^{[n]} g \| ^2  \leq  \underbrace{\| (X^*) ^{[n]} (g - g _{K+1} ) \| ^2}_{\leq \eps ^2} + \underbrace{\| (X^*) ^{[n]} \mr{e} ^* h_{K+1} \| ^2}_{\rightarrow 0}. $$ This proves that 
$$ \lim _{n\rightarrow \infty} \| (X^*) ^{[n]} g \| ^2 < \eps ^2, $$ for any $\eps >0$ so that 
$ \| (X^*) ^{[n]} g \| ^2 \rightarrow 0$, proving the claim.
\end{proof}

\begin{proof}{ (of Theorem \ref{NCdBB}) }
First suppose that $\scr{M} = \scr{M} ^L (F)$. Then any $x \in \scr{M}$ has the form $x= F(L) f$ for some $f \in \bH ^2 _d \otimes \mc{F}$, and 
\ba \| \mr{e} x \| ^2 _{\bH ^2 _d \otimes \J} & = & \| F(L) f \| ^2 _{\bH ^2 _d \otimes \J } \nn \\
& \leq & \| F(L) \| ^2 \| P _{\ker{F}} ^\perp f \| ^2 _{\bH ^2 _d \otimes \mc{F}} \nn \\
& = & \| F(L) \| ^2 \| x \| ^2 _{F}. \nn \ea As proven in Lemma \ref{rowcontract}, $\scr{M} ^L (F)$ is $R\otimes I_\J -$invariant and $X_F = (R \otimes I_\J ) | _{\scr{M} ^L (F)}$ is a row contraction.

Conversely suppose that $\scr{M} \bc \bH ^2 _d \otimes \J$ obeys the hypotheses of the theorem. 
Let $X := R \otimes \J | _{\scr{M}} $. By Lemma \ref{pure}, $X$ is pure, so that the minimal row isometric dilation, $\hat{X}$, of $X$ is unitarily equivalent to copies of $R$, $\hat{X} \simeq R \otimes I _\mc{F}$ \cite[Proposition 2.3]{Pop-dil}. Let $\hat{\scr{M}}$ be the Hilbert space on which $\hat{X}$ acts. Let $U : \bH ^2 _d \otimes \mc{F} \rightarrow \hat{\scr{M}}$ be the unitary so that $U R \otimes I_{\mc{F}} (U^* \otimes I_d ) = \hat{X}$, and define the bounded operator 
$$ F:=   \mr{e} P _\scr{M} U : \bH ^2 _d \otimes \mc{F} \rightarrow \bH ^2 _d \otimes \J, $$ where $P_\scr{M}$ is the orthogonal projection of $\hat{\scr{M}}$ onto $\scr{M}$ and $\mr{e} : \scr{M}  \hookrightarrow \bH ^2 _d \otimes \J$ is the bounded embedding so that $\| F \| \leq \| \mr{e} \|$. Then observe that 
\ba F R \otimes I _{\mc{F}} & = & \mr{e} P _\scr{M} \hat{X}  (U \otimes I_d )  \nn \\
& = & \mr{e} R \otimes I_{\mc{J}} | _{\scr{M}} P_{\scr{M}} (U \otimes I_d ) \nn \\
& = & R \otimes I_\J F. \nn \ea By \cite[Theorem 1.2]{DP-inv}, $F = F(L) \in  \bH ^\infty _d \otimes \L (\mc{F} , \J )$. Observe that $U \ker{F (L)} ^\perp = \scr{M}$: If $h \in \bH ^2 _d \otimes \mc{F}$ and $F(L) h = 0$, then  
$$ 0  =  Fh  = \mr{e} P_\scr{M} U h, $$ which happens if and only if $P_\scr{M} U h = 0 $ since $\mr{e}$ is injective.
This proves that $U \ker{F(L)} \subseteq \scr{M} ^\perp$. Conversely if $m_\perp \in \scr{M} ^\perp$ then
$$ F U^* m_\perp = \mr{e} P_\scr{M} U U^* m_\perp =0, $$ so that 
$$ \ker{F(L)} = U^* U \ker{F(L)} \subseteq U^* \scr{M} ^\perp \subseteq \ker{F(L)}, $$ and the claim follows. 

Next, observe that $\scr{M} = \scr{M} ^L (F)$: Given any $g \in \scr{M}$, let $h := U^* g  \in \bH ^2 _d \otimes \mc{F}$, and note that 
$$ (F (L) h ) (Z) = (\mr{e} P_\scr{M} U U^* g) (Z) = g(Z), $$ for any $Z \in \B ^d _\N$. This shows that $\scr{M} \subseteq \scr{M} ^L (F)$. Conversely given any $F(L) h \in \scr{M} ^L (F)$, for $h \in \bH ^2 _d \otimes \mc{F}$, we can assume, without loss in generality that $h = P _{\ker{F} } ^\perp h$ so that $g= Uh \in \scr{M}$. Then, as above, 
$$ (F(L) h) (Z) = (\mr{e} P_\scr{M} U h ) (Z) = g(Z), $$ so that $\scr{M} = \scr{M} ^L (F)$ as vector spaces.

We claim that the norms of $\scr{M}$ and $\scr{M} ^L (F)$ are the same, so that $\scr{M} = \scr{M} ^L (F)$ isometrically:
Given any $g \in \scr{M} = \scr{M} ^L (F)$,
$$ \| g \| ^2 _\scr{M} = \| U^* g \| _{\bH ^2 _d \otimes \mc{F}}, $$ since $U^*$ is unitary, and 
\ba \| g \| ^2 _{\scr{M} ^L (F)} & = & \| \mr{e} P_\scr{M} U U^* g \| ^2 _{\scr{M} ^L (F)} \nn \\
& = & \| F(L) U^* g \| ^2 _{\scr{M} ^L (F)} \nn \\
& = & \|P _{\ker{F}} ^\perp U^* g \| ^2 _{\bH ^2 _d \otimes \mc{F}} \nn \\
& =& \| U^* g \| ^2 _{\bH ^2 _d \otimes \mc{F}} = \| g \| ^2 _{\scr{M}}. \nn \ea 

Finally, by construction, for any $x \in \ker{F(L) } ^\perp$, 
\ba \| F(L) x \| ^2 _{\bH ^2 _d \otimes \J } & = & \| \mr{e} P_\scr{M} U x \| ^2 \nn \\
& = & \| \mr{e} Ux \| ^2, \nn \ea since $U  \ker{F(L)} ^\perp = \scr{M}$. This proves that $\| F(L) \| = \| \mr {e} \|$.
\end{proof}

\begin{cor} \label{Multfactor}
If $\scr{M} \bc \bH ^2 _d \otimes \J$ satisfies the conditions of the NC de Branges-Beurling Theorem, then $\scr{M} = \scr{M} ^L (F)$ for 
$$ F(L) := \mr{e} P_\scr{M} U, $$ where $\mr{e} : \scr{M} \hookrightarrow \bH ^2 _d \otimes \J$ is the bounded embedding, $P_\scr{M} : \hat{\scr{M}} \rightarrow \scr{M}$ is orthogonal projection,
$\hat{\scr{M}} \subseteq \scr{M}$ is the Hilbert space of the minimal row isometric dilation $\hat{X} _F$ of $X_F := (R \otimes I_\J) | _{\scr{M}}$, and $U:\bH ^2 _d \otimes \mc{F} \rightarrow \hat{\scr{M}}$ is the unitary intertwining $R \otimes I_{\mc{F}}$ and $\hat{X} _F$. 
\end{cor}

\begin{remark} \label{ranUF}
Further observe that $F(L) = \mr{e} _F P _{\scr{M} ^L (F )} U_F$, where 
$U_F : \bH ^2 _d \otimes \mc{F} \rightarrow \hat{\scr{M}} ^L (F)$ intertwines $R \otimes I _\mc{F}$ with the minimal row isometric dilation of $(X_F , \scr{M} ^L (F) )$ where $U_F \ker{F(L)} ^\perp = \scr{M} ^L (F)$. Hence, if $f \in \ker{F(L)} ^\perp$ then $U_F f = F(L) f $, and $U_F \ker{F(L)} = \hat{\scr{M}} ^L (F) \ominus \scr{M} ^L (F)$.
\end{remark}

\subsection*{Aside: Complementary Spaces}

In \cite{dB-ss}, de Branges proves results using a `complementary viewpoint'. Namely, instead of working with shift-invariant subspaces contractively-contained in $H^2$, he studies shift co-invariant subspaces contractively contained in $H^2$ which obey an `inequality for difference quotients'. Such subspaces are complementary to shift-invariant subspaces in the following sense:

\begin{defn}{ (\cite{dBss},\cite[Section 16.9]{FMHb2})}
    Given a Hilbert space $\mc{H}$, and a Hilbert space, $\scr{M}$, contractively contained in $\H$, the \emph{complementary space} to $\scr{M}$ is the Hilbert space $\scr{H} \subseteq \H$ defined as the set of all $h \in \H$ so that 
$$ \| h \| ^2 _\scr{H} := \sup _{m \in \scr{M}} \left( \| h +m \| ^2 _{\mc{H}} - \| m \| ^2 _{\scr{M}} \right) < +\infty. $$ 
\end{defn}
\begin{lemma}
Suppose that $\scr{M} _1 , \scr{M} _2$ are contractively contained in a Hilbert space, $\mc{H}$. Then $\scr{M} _1 \cc \scr{M} _2$ if and only if their complementary spaces obey $\scr{H} _2 \cc \scr{H} _1$.
\end{lemma}
\begin{defn}
    A linear subspace $\scr{H} \subseteq H^2 (\B ^d _N ) $ which is contractively contained in the NC Hardy Space and co-invariant for the right free shifts, is said 
    to obey the \emph{inequality for (right) difference quotients} if:
$$ \| R ^* h \| _{\scr{H} \otimes \C ^d } ^2 \leq \| h \| ^2 _\scr{H}  - | h(0) | ^2; \quad \quad h \in \scr{H}. $$
\end{defn}

\begin{lemma}
Let $\scr{M} \cc \bH ^2 _d$ be a HIlbert space contractively contained in the NC Hardy space. Then $\scr{M}$ is $R-$invariant, and $R|_{\scr{M}}$ is a row contraction, if and only if the complementary space $\scr{H} \cc \bH ^2 _d$ is $R-$coinvariant and obeys the inequality for NC difference quotients.
\end{lemma}

\begin{proof}
    Given such a $\scr{M}$, suppose $h \in \scr{H}$. Then,
\ba \infty & > & \| h \| ^2 _{\scr{H}} = \sup _{m \in \scr{M}} \left( \| h + m \| ^2 _{\bH ^2} - \| m \| ^2 _\scr{M} \right) \nn \\
& \geq & \sup _{\mbf{m}} \left( \| h + R \mbf{m} \| ^2 _{\bH ^2} - \| R \mbf{m} \| ^2 _{\scr{M}} \right) \nn \\
& \geq & \sup \left( \| h \| ^2 _{\bH ^2}   - 2 \re{\ip{R^* h}{\mbf{m}} _{\bH ^2}} + \| \mbf{m} \| ^2 _{\bH ^2} - \| \mbf{m} \| ^2 _\scr{M} \right) \nn \\
& = & \sup \left( \| R^* h \| ^2 _{\bH ^2} + | h(0) | ^2 - 2\re{\ip{R^*h}{\mbf{m}}_{\bH ^2}} + \| \mbf{m} \| ^2 _{\bH ^2} - \| \mbf{m} \| ^2 _\scr{M} \right) \nn \\ 
& = & \sup _\mbf{m} \left( \| R^* h + \mbf{m} \| ^2 _{\bH ^2} - \| \mbf{m} \| _\scr{M} ^2 + | h (0) | ^2 \right) \nn \\
& = & \| R^* h \| ^2 _{\scr{H}} + |h (0) | ^2. \nn \ea 
This proves that $\scr{H}$ is $R-$coinvariant and obeys the inequality for NC difference quotients. The converse is similarly easy to verify. 
\end{proof}

\subsection*{Right ideals}

Davidson and Pitts in \cite{DP-alg} proved that there is a lattice isomorphism between the lattice of weak operator topology (WOT)-closed right ideals of $\mult$ and the closed, $R-$invariant subspaces of $\hardy$. The correspondence is given by $J \mapsto \overline{J \cdot 1}$. The inverse map is $\iota(\sM) = \left\{ G(L) \in \mult \mid G(L) \cdot 1 \in \M \right\}$. 

Given a row $F(L) \colon \bH ^2_d \otimes \J \to \bH ^2_d$, we define $J_F$ to be the algebraic ideal generated by the entries of $F(L)$. We will denote by $\overline{J_F}$ the norm closure of $J_F$, i.e, the norm closed right ideal generated by the entries of $F(L)$.

\begin{prop}
Let $F(L) \colon \bH ^2_d \otimes \J \to \bH ^2_d$ and $G(L) \colon \bH ^2_d \otimes \K \to \bH ^2_d$ be two rows. If $\ran{F(L)} \subset \ran{G(L}$, then $\overline{J_F} \subset \overline{J_G}$. Conversely, if $J_F \subset J_G$, then $\ran{F(L)} = \ran{G(L)}$.
\end{prop}
\begin{proof}
Since $\ran{G(L)} \subset \ran{F(L)}$, by the Douglas factorization lemma, we have that $F(L) = G(L) B(L)$, for some operator-valued multiplier $B(L)$. However, this implies that the entries of $F(L)$ are elements of $\overline{J_G}$.

Conversely, Let us write $n = \dim{\J}$ and $m = \dim{\K}$. If $J_F \subset J_G$, then, for every $1 \leq j \leq n$, there exist $\{A_{ij}(L)\}_{i=1}^m \subset \mult$, all but finitely many of which are zero, such that 
\[
F_j(L) = \sum_{i = 1}^m G_i(L) A_{ij}(L).
\]
Since multiplying entries of $F(L)$ by scalars won't change either $J_F$ or $\ran{F(L)}$, we can multiply each $F_j(L)$ by $\frac{1}{2^j}$. This will lead to multiplication of the $A_{ij}(L)$ by the same factor. Thus, the operator matrix (perhaps infinite) $A(L) = \left[ A_{ij}(L) \right]$ defines a bounded operator from $\bH ^2_d \otimes \J$ to $\bH ^2_d \otimes \K$. Furthermore, $F(L) = G(L) A(L)$ and we conclude that $\ran{F(L)} \subset \ran{G(L)}$.
\end{proof}

\begin{remark}
Note that, if $n = \dim{\J}, m = \dim{\K} < \infty$, then $\ran{F(L)} \subset \ran{G(L)}$ implies that there exists an finite matrix $A(L)$, such that $F(L) = G(L) A(L)$. Therefore, $J_F \subset J_G$.
\end{remark}

\section{Lattice operations on CPNC kernels} \label{sec:lattice_kernels}

The results of this section are inspired by \cite{Aron-RKHS,GheOkut}. One can define two lattice operations, $\vee, \wedge$ on NC reproducing kernel Hilbert spaces on the same NC set as follows:

\begin{defn}
Let $K, k$ be two CPNC kernels defined on the same NC set. $\H _{nc} (K) \vee \H _{nc} (k)$, or $\H _{nc} (K \vee k)$, is the NC-RKHS corresponding to the CPNC kernel:
$$ (K\vee k ) (Z, W) := K(Z,W) + k (Z,W). $$ 
The NC-RKHS $\H _{nc} (K ) \wedge \H _{nc} (k) = \H _{nc} (K \wedge k)$ is then defined as 
$$ \H _{nc} (K \wedge k) := \H _{nc} (K) \bigcap \H _{nc} (k), $$ with the norm:
$$ \| \cdot \| ^2 _{K \wedge k} := \| \cdot \| ^2 _K + \| \cdot \| ^2 _k. $$ 
\end{defn}
Observe that $\H _{nc} (K \wedge k)$ embeds contractively in both $\H _{nc} (K)$ and $\H _{nc} (k)$.

\begin{thm} \label{sumintrkhs}
Define isometries $U _\vee , U _\wedge : \H _{nc} (K \vee k), \H _{nc} (K \wedge k) \rightarrow \H _{nc} (K) \oplus \H _{nc} (k)$ by:
$$ U_\vee (K\vee k)_Z := K_Z \oplus k_Z, \quad \mbox{and} $$ $$ U_\wedge h = h \oplus -h. $$ Then 
$$ \H _{nc} (K) \oplus \H _{nc} (k) = U_\vee \H _{nc} (K \vee k) \oplus U_\wedge \H _{nc} (K \wedge k), \quad \mbox{and}$$  
$$ (K \wedge k )_Z = \frac{1}{2} U_\wedge ^* (K_Z \oplus -k_z) =  U_\wedge ^* (K_Z \oplus 0) =  U_\wedge ^* (0 \oplus -k_Z). $$ 
\end{thm}
\begin{proof}
The claims are easy to verify. In particular, it is clear that $\ran{U_\vee} \subseteq \ran{U_\wedge} ^\perp$. Conversely if $f \oplus g  \perp \ran{U_\vee}$ then for all $Z$,
\ba 0 & = & \ip{K_Z \oplus k_Z}{f \oplus g} \nn \\
& = & f(Z) + g(Z). \nn \ea This proves that $f(Z) = - g(Z)$ so that $f\oplus g = f \oplus -f \in \ran{U_\wedge}$, and we conclude that $\ran{U_\vee} ^\perp = \ran{U_\wedge} $. 
To prove the identity for $(K \wedge k) _Z$, observe that for any $h \in \H _{nc} (K\wedge k)$, 
\ba h(Z) & = & \left( (K \wedge k) _Z \right)^* h \nn \\
& = & \frac{1}{2} \left( K_Z \oplus -k_Z \right) ^* (h \oplus -h) \nn \\
& = & \frac{1}{2} \left( K_Z \oplus k_z \right) ^* U_\wedge h \nn \\
& = & \frac{1}{2} \left(  U_\wedge ^* (K_Z \oplus k_z)  \right) ^* h. \nn \ea 
Furthermore,
$$ K_Z \oplus 0 = \frac{1}{2} (K_Z \oplus k_Z ) + \frac{1}{2} (K_Z \oplus - k_Z), $$ so that
$$ P_\wedge (K_Z \oplus 0) = \frac{1}{2} P_\wedge (K_Z \oplus -k_Z), $$ and 
$$ U_\wedge ^* (K_Z \oplus 0) = \frac{1}{2} U_\wedge ^* (K_Z \oplus -k_Z) = (K\wedge k)_Z. $$ 
\end{proof}

\section{Lattice operations on NC left multipliers} \label{sec:lattice_mult}

\begin{cor}
Given operator-valued left NC multipliers, $F (L) \in \bH ^\infty _d \otimes \mc{L} (\mc{F} , \J) $ and 
$G(L) \in \bH ^\infty _d \otimes \mc{L} (\mc{G} , \J ), $ with the same coefficient range space, $\J$, let $\scr{M} := \scr{M} ^L (F) \wedge \scr{M} ^L (G) $. Then there is an operator-valued left multiplier $H(L) \in \bH ^\infty _d \otimes \L (\H, \J )$ with $\mr{dim} (\cH ) \leq \mr{dim} (\cF) + \mr{dim} (\cG)$ so that $\scr{M} = \scr{M} ^L (H)$,
$$ \ran{H(L)} = \ran{F(L)} \bigcap \ran{G(L)}, $$ and $\| H(L) \| \leq \max \{ \| F(L) \| , \| G( L) \| \} $.
\end{cor}
\begin{proof}
This is immediate by the definition of the NC-RKHS $\scr{M} ^L (F) \wedge \scr{M} ^L (G)$ and the NC de Branges-Beurling theorem. 
\end{proof}

Let $F(L), G(L)$ be operator-valued left NC multipliers with the same coefficient range space, $\J$, as in the statement of the corollary above. Then we can define two new operator-valued left multipliers, $F \vee G$ and $F \wedge G$ with the same coefficient range space, $\J$ as follows.

\begin{defn}
Given $F, G, H$ as above, $F \vee G := (F , G) \in \bH ^\infty _d \otimes \L (\mc{F} \oplus \mc{G}, \J )$ and $F \wedge G := H \in \bH ^\infty _d \otimes \L ( \H, \J )$, where $\scr{M} ^L (H) := \scr{M} ^L (F) \wedge \scr{M} ^L (G)$ as above. 
\end{defn}

\begin{cor} \label{domuvee}
Given $F, G$, $F \vee G$ and $F \wedge G$ as above,
$$\scr{M} ^L (F) \vee \scr{M} ^L (G) = \scr{M} ^L (F \vee G), $$ and 
$$ \scr{M} ^L (F) \wedge \scr{M} ^L (G) = \scr{M} ^L (F \wedge G). $$ In particular,
$$ \scr{M} ^L (F \oplus G) = U _\vee \scr{M} ^L (F \vee G) \oplus U_\wedge \scr{M} ^L (F \wedge G), $$ where the isometries $U_\vee, U _\wedge$ are as defined in Theorem \ref{sumintrkhs}. Moreover, in this case, since $F \vee G = (F, G)$, $U_\vee$ can be defined on the dense set $$ \ran{(F\vee G) (L) (F\vee G) (L) ^* } \subseteq \scr{M} ^L (F \vee G)$$ by:
\ba U_\vee (F \vee G) (L) (F \vee G) (L) ^* x & = & U_\vee \left( F(L) , G(L) \right) \bpm F(L) ^* \\ G(L) ^* \epm x \nn \\
& = & U_\vee (F(L) F(L) ^* x + G(L) G(L) ^* x ) \nn \\
& = & F(L) F(L) ^* x \oplus G(L) G(L) ^* x; \quad x \in \bH ^2 _d \otimes \J.  \nn \ea 
\end{cor}
\begin{proof}
The only thing to prove is the final statment regarding the action of $U_\vee$. Calculate: For any $x \in \bH ^2 _d \otimes \J$,
\ba 
& & \| F(L) F(L) ^* x \oplus  G(L) G(L) ^* x \| _{F\oplus G} ^2  =  \| F(L) F(L) ^* x \| ^2 _F + \| G(L) G(L) ^* x \| ^2 _G \nn \\
& =& \| P_{\ker{F}} ^\perp F(L) ^* x \| ^2 _{\bH ^2 _d \otimes \mc{F}} + \| P _{\ker{G}} ^\perp G(L) ^* x \| ^2 _{\bH ^2 _d \otimes \mc{G}} \nn \\
& =& \|  F(L) ^* x \| ^2 _{\bH ^2 _d \otimes \mc{F}} + \|  G(L) ^* x \| ^2 _{\bH ^2 _d \otimes \mc{G}} \quad \quad \mbox{(Since $\ran{A ^* } \subseteq \ker{A} ^\perp$.)}\nn \ea
\ba & = & \left\|  \bpm F(L) ^* \\  G(L) ^* \epm x  \right\| ^2 _{\bH ^2 _d \otimes (\mc{F} \oplus \mc{G} )}  \nn \\
& = & \left\| P _{\ker{ (F,G) }} ^\perp  \bpm F(L) ^* \\  G(L) ^* \epm x  \right\| ^2 _{\bH ^2 _d \otimes (\mc{F} \oplus \mc{G} )} \quad \quad \mbox{(Same reason as above)} \nn \\
& = & \left\| \left( F(L) , G(L) \right)  \bpm F(L) ^* \\  G(L) ^* \epm x  \right\| ^2 _{\scr{M} ^L (F \vee G)}. \nn
\ea
This proves that $U_\vee$ is an isometry, and it agrees with the previous definition of $U_\vee$ since the kernel maps for $\scr{M} ^L (F)$ are:
$$ K^F _Z = F(L) F(L) ^* K_Z, $$ where $K$ denotes the CPNC Szeg\"{o} kernel for the NC Hardy space. Also note that for any bounded linear map between Hilbert spaces, $\ran{A^*}$ is dense in $\ker{A} ^\perp$.
\end{proof}

\begin{remark}
In particular, recall that 
$$ U_\vee \left( F(L), G(L) \right) \bpm F(L) ^* \\ G(L) ^* \epm x = \bpm F(L) F(L) ^* x \\ G(L) G(L) ^* x \epm. $$
Hence 
\ba  \bpm X_F ^* & 0 \\ 0 & X_G ^* \epm \bpm F(L) F(L) ^* x \\ G(L) G(L) ^* x \epm 
& = &  \bpm F(L) (R ^*  \otimes I_\mc{F} ) F(L) ^* x \\ G(L) (R ^* \otimes I_{\mc{G}} ) G(L) ^*x  \epm \nn \\
& = & \bpm F(L) F(L) ^* R ^* \otimes I_{\mc{J}} x \\ G(L) G(L) ^* R ^* \otimes I_{\mc{J}} x \epm, \nn \ea 
so that $\ran{U_\vee}$ is indeed $X_F \oplus X_G - $co-invariant.
\end{remark}

\begin{remark}
If $H = F \wedge G$, observe that 
$$ U_\wedge H(L) x = H(L) x \oplus -H(L) x, $$ so that 
$$ \ran{U_\wedge} = \scr{M} ^L \bpm H \\ - H \epm, $$ and that this subspace is $X_F \oplus X_G-$invariant. We could have instead defined
$$ U_\vee \left( F(L) , G(L) \right) \bpm F(L) ^* \\ G(L) ^* \epm x := F(L) F(L) ^* x \oplus - G(L) G(L) ^* x, $$ and $ U_\wedge h := h \oplus h.$ In this case we would still have that 
$$ \scr{M} ^L (F \oplus G) = \ran{U_\vee} \oplus \ran{U _\wedge}, $$ and 
$ \ran{U_\wedge} = \scr{M} ^L \bsm H \\ H \esm. $
\end{remark}

\subsection*{The lattice of operator-valued left multipliers}

Consider the set $\bH ^\infty _d \otimes \mc{L} ( \cdot, \cJ )$ of all operator-valued left multipliers with range contained in $\bH ^2 _d \otimes \J$. We define an equivalence relation on $\bH ^\infty _d \otimes \mc{L} (\cdot , \cJ )$ by: $$ F(L) \sim G(L) \quad \mbox{if} \quad F(L) = G(L) C(L), $$ for some invertible operator-valued left multiplier $C(L) \in \bH ^\infty _d \otimes \mc{L} (\mc{F}, \mc{G} )$, where $F (L) \in \bH ^\infty _d \otimes \mc{L} (\mc{F} , \mc{J} )$, and $G (L) \in \bH ^\infty _d \otimes \mc{L} (\mc{G} , \mc{J} )$.

\begin{thm}
The set $ \displaystyle \bH ^\infty _d \otimes \mc{L} ( \cdot, \cJ ) / \sim$ is a bounded general lattice.
\end{thm}
\begin{proof}
We need to check that the set of all operator-valued left multipliers with common coefficient range space $\cJ$ modulo right multiplication by invertible operator-valued left multipliers satisfies the axioms of a bounded general lattice.

First, the commutative laws clearly hold: Given any $F, G$ with common coefficient range space $\cJ$ we have $F \vee G = (F , G) \sim (G,F) = G \vee F$ (via a constant unitary permutation multiplier) and $F \wedge G \sim G \wedge F$. The associative laws are also straightforward: $F \vee (G \vee H) = (F, G\vee H) = (F, G, H) = (F \vee G) \vee H$. Similarly, consider $F \wedge (G \wedge H)$. Then, as vector spaces,
\ba \scr{M} ^L \left( F \wedge (G \wedge H) \right) & = & \scr{M} ^L (F) \bigcap \scr{M} ^L (G \wedge H) \nn \\
& = & \scr{M} ^L (F) \bigcap \scr{M} ^L (G) \bigcap \scr{M} ^L (H) \nn \\
& = & \scr{M} ^L \left( (F \wedge G) \wedge H \right). \nn \ea
Moreover, the norms are the same:
\ba \| h \| ^2 _{F \wedge (G \wedge H )} & = & \| h \| ^2 _F + \| h \| ^2 _{G \wedge H} \nn \\
& = & \| h \| ^2 _F + \| h \| ^2 _G + \| h \| ^2 _H \nn \\
& = & \| h \| ^2 _{(F\wedge G) \wedge H}, \nn \ea
and the associative law for $\wedge$ follows. The absorption laws are less trivial: Consider $F \wedge (F \vee G)$. Again, as vector spaces,
\ba \scr{M} ^L \left( F \wedge (F \vee G) \right) & = & \scr{M} ^L (F) \bigcap \scr{M} ^L (F, G) \nn \\
& = & \scr{M} ^L (F) \bigcap \left( \scr{M} ^L (F) \bigcup \scr{M} ^L (G) \right) \nn \\
& = & \scr{M} ^L (F). \nn \ea 
Now compare the norms: 
\ba \| f \| ^2 _{F} & \leq & \underbrace{\| f \| ^2 _{F} + \| f \| ^2 _{F \vee G} }_{ \| f \| ^2 _{F \wedge (F \vee G)}} \nn \\
& \leq & \| f \| ^2 _{F} + \| f \| ^2 _F = 2 \| f \| ^2 _F, \nn \ea since $\scr{M} ^L (F) \cc \scr{M} ^L (F,G)$. By the Douglas factorization property, it follows that if $T(L) := \left( F \wedge (F \vee G) \right) (L)$, that there is a contractive operator-valued left multiplier $C(L)$, and a bounded operator-valued left multiplier $D(L)$ of norm at most $\sqrt{2}$ so that 
$$ F(L) = T(L) C(L), \quad \mbox{and} \quad T(L) = F(L) D(L). $$ One can further check that $C(L) = D(L) ^{-1}$, so that $F \sim F \wedge (F \vee G)$. To see that $C(L) = D(L) ^{-1}$, consider the factorization from Corollary \ref{Multfactor}: We have that
$$ F(L) = \mr{e} _F P_F U_F, \quad \mbox{and} \quad T(L) = \mr{e} _T P_T U_T.$$ Here, $\mr{e} _F : \scr{M} ^L (F ) \hookrightarrow \bH ^2 _d \otimes \cJ$ is the bounded embedding with $\| \mr{e} _F \| = \| F(L) \|$, $U_F : \bH ^2 _d \otimes \mc{F} \rightarrow \hat{\scr{M}} ^L (F)$ is the unitary intertwining $R \otimes I _{\mc{F}}$ with the minimal row isometric dilation, $\hat{X} _F $ of $X_F := \left. R \otimes I_{\cJ} \right| _{\scr{M} ^L (F)}$, and $P_F : \hat{\scr{M}} ^L (F) \rightarrow \scr{M} ^L (F)$ is orthogonal projection from the space of the minimal row isometric dilation, $\hat{\scr{M} } ^L (F)$ onto $\scr{M} ^L (F)$. 
Since the norms of $\scr{M} ^L (F)$ and $\scr{M} ^L (T)$ are equivalent, we further have that the embedding $\mr{e} : \scr{M} ^L (F) \hookrightarrow \scr{M} ^L (T)$ is contractive and invertible and the inverse embedding $\mr{e} ^{-1} : \scr{M} ^L (T) \hookrightarrow \scr{M} ^L (F)$ has norm at most $\sqrt{2}$. Observe that 
$$ \mr{e} X_F ^\alpha = X_T ^\alpha \mr{e}, \quad \mbox{and} \quad \mr{e} ^{-1} X_T ^\alpha = X_F ^\alpha \mr{e} ^{-1}. $$ By commutant lifting there exist `lifted embeddings' $$ \hat{\mr{e}} : \hat{\scr{M}} ^L (F) \rightarrow \hat{\scr{M}} ^L (T), \quad  \left. \hat{\mr{e}} ^* \right| _{\scr{M} ^L (T)} = \mr{e} ^*, \quad \| \hat{\mr{e}} \| = \| \mr{e} \|, $$ so that 
$$ \hat{\mr{e}} \hat{X} _F ^\alpha = \hat{X} _T ^\alpha \hat{\mr{e}}. $$ Similarly let $\hat{\mr{e}} _{-1}$ be the lift of $\mr{e} ^{-1}$, this has norm at most $\sqrt{2}$. Consider $C := U_T ^* \hat{\mr{e}} U_F$, this intertwines the right free shifts, 
$$ C R ^\alpha \otimes I_{\mc{F}} = R ^\alpha \otimes I _{\mc{T}} C, $$ so that $C =C(L) \in \bH ^\infty _d \otimes \mc{L} (\mc{F} , \mc{T} )$ is contractive. Further observe that 
\ba T(L) C(L) & = & \mr{e} _T P _T U_T U_T ^* \hat{\mr{e}} U_F \nn \\
& = & \mr{e} _T \underbrace{P_T \hat{\mr{e}} P_F}_{= e} U_F \nn \\
& = & \underbrace{\mr{e} _T \mr{e}}_{=\mr{e} _F} P_F U_F = F(L). \nn \ea
Similarly, setting $D(L) := U_F ^* \hat{\mr{e}} _{-1} U_T$ we obtain that $F(L) D(L) = T(L)$. Finally, we claim that $\hat{\mr{e}} _{-1} = \hat{\mr{e}} ^{-1}$ so that $C(L) = D(L) ^{-1}$. Indeed we have that 
$$ \hat{\scr{M}} ^L (F) = \bigvee \hat{X} _F ^\alpha \scr{M} ^L (F), $$ and similarly for $\hat{\scr{M}} ^L (T)$. Hence, it suffices to check that for any $f, g \in \scr{M} ^L (F)$ we have that
$$ \ip{\hat{X} _F ^\beta g}{\hat{\mr{e}} _{-1} \hat{\mr{e}} \hat{X} _F ^\alpha f} =  \ip{\hat{X} _F ^\beta g}{\hat{X} _F ^\alpha f},  $$ where say $\alpha = \beta \ga$. Indeed, in this case
\ba \ip{\hat{X} _F ^\beta g}{\hat{\mr{e}} _{-1} \hat{\mr{e}} \hat{X} _F ^\alpha f} & = & \ip{\underbrace{(\hat{X} _F ^\ga) ^* g}_{\in \scr{M} ^L (F)}}{\hat{\mr{e}}_{-1} \hat{\mr{e}} f } \nn \\
& = & \ip{(\hat{X} _F ^\ga) ^* g}{P_F \hat{\mr{e}}_{-1} \hat{\mr{e}} P_F f} \nn \\
& = & \ip{(\hat{X} _F ^\ga) ^* g}{P_F \hat{\mr{e}}_{-1} P_T \hat{\mr{e}} P_F f} \nn \\
& = & \ip{(\hat{X} _F ^\ga) ^* g}{\mr{e} ^{-1} \mr{e} f} \nn \\
& = & \ip{\hat{X} _F ^\beta g}{\hat{X} _F ^\alpha f}. \nn \ea
Hence $T(L) = F(L) C(L)$ for an invertible left multiplier $C(L) = D(L) ^{-1}$, and this verifies this absorption law for the quotient set. Verification of the second absorption law $ F \vee (F \wedge G)=F$ is similar and omitted. Finally, this is a bounded general lattice since it has a `bottom element', the zero multiplier on $\bH ^2 _d \otimes \J$ which is an identity for the join operation $\vee$, and it has a `top element', the constant identity multiplier $I_{\bH ^2_d} \otimes I_{\cJ}$, which is an identity for the meet operation $\wedge$.
\end{proof}

\subsection*{Factorization of $F\wedge G$}

Let $\mr{e} _{F} , \mr{e} _G$ be the bounded embeddings of $\scr{M} ^L (F), \scr{M} ^L (G)$, respectively, into $\bH ^2 _d \otimes \J$. Similarly, let $X_F := (R \otimes I_\J) | _{\scr{M} ^L (F)}$, and $U_F : \bH ^2 _d \otimes \mc{F} \rightarrow \hat{\scr{M}} ^L (F)$, where $$ \hat{\scr{M}} ^L (F) := \bigvee \hat{X} _F ^\alpha \scr{M} ^L (F), $$ is the Hilbert space of the minimal row-isometric dilation, $\hat{X} _F$, of $X_F$, $P_F : \hat{\scr{M}} ^L (F) \rightarrow \scr{M} ^L (F)$ is orthogonal projection, and 
$ U_F (R ^\alpha  \otimes I  _{\mc{F}} ) = \hat{X} _F ^\alpha U_F, $ so that 
$ F (L) = \mr{e} _F P _F U_F$, as in the proof of the NC de Branges-Beurling theorem. 

\begin{thm} \label{Hid}
Given $F,G$ as above, and $H (L) = F(L) \wedge G(L) \in \bH ^\infty _d \otimes \mc{L} (\mc{H} , \mc{J} )$, there is a left inner $\Ga (L) : \bH ^2 _d \otimes \H \rightarrow \bH ^2 _d \otimes (\mc{F} \oplus \mc{G} )$ with $\ran{\Ga (L)} \subseteq \ker{(F \vee G) (L)}$ so that:
\ba \bpm H(L) \\ - H(L) \epm & = & \bpm F(L) & 0 \\ 0 & G(L) \epm \Ga (L). \ea
\end{thm}
\begin{cor}
If $F,G$ are inner, then so is $H = F \wedge G$.
\end{cor}
\begin{proof}{(of Theorem \ref{Hid})}
Let $\scr{M} := \scr{M} ^L \bsm H \\ - H \esm$.  First observe that $\scr{M} = U_\wedge \scr{M} ^L (H)$. Further observe that $\ran{P_\wedge }$, where $P_{\wedge} = P_{\bsm H \\ - H \esm}$ is $X_F \oplus X_G = \left( R \otimes I_\J \otimes \C ^2 \right) | _{\scr{M} ^L (F \oplus G)} -$invariant, and that 
$$ X _F \oplus X_G | _{\scr{M} ^L \bsm H \\ - H \esm} = X_{\bsm H \\ -H \esm}=: X. $$
It follows that $\hat{X} _F \oplus \hat{X} _G$ is a row isometric dilation of $$X  = \left. (R \otimes I_\J \otimes \C ^2 ) \right| _{\scr{M} ^L \bsm H \\ -H \esm }. $$ 
Indeed, for any $\alpha \in \F ^d$, since $P_\wedge \leq P_F \oplus P_G$,
\ba P_{\wedge} \hat{X} _F ^\alpha \oplus \hat{X} _G ^\alpha P_{\wedge} & = & P_\wedge (P_F \hat{X} _F ^\alpha P_F \oplus P_G \hat{X} _G ^\alpha P_G) P _\wedge \nn \\
& = & P_\wedge (X_F ^\alpha \oplus X_G ^\alpha) P_\wedge \nn \\
& = & X ^\alpha P_\wedge, \nn \ea and this proves that $\hat{X} _F \oplus \hat{X} _G$ is a row isometric dilation of $X$. Setting
$$ \hat{\scr{M}} := \bigvee _{\alpha \in \F ^d} \left( \hat{X} _F ^\alpha \oplus \hat{X} _G ^\alpha  \right) \scr{M} ^L \bpm H \\ - H \epm, $$ this is $\hat{X} _F \oplus \hat{X} _G -$invariant, and 
$$ \hat{X} := \hat{X} _{\bsm H \\ - H \esm } = \hat{X} _F \oplus \hat{X} _G | _{\hat{\scr{M}}}, $$ is the minimal row isometric dilation of $X = X_{\bsm H \\ - H \esm }$. In particular, it follows by \cite[Theorem 2.1]{Pop-dil} (see also Subsection \ref{minimal} and Lemma \ref{minunique}), that $\scr{M} ^L \bsm H \\ -H \esm$ is $\hat{X}-$coinvariant, and 
$$ \hat{X} ^* | _{\scr{M} ^L \bsm H \\ -H \esm} = X^*. $$ 

Since $\hat{X}$ is the minimal row isometric dilation of $X_{\bsm H \\ - H \esm} =X$, we have, by the NC de Branges-Beurling theorem that,
$$ \bpm H (L) \\ - H(L) \epm = \mr{e} _F \oplus \mr{e} _G P _\wedge U_H, $$ where 
$U_H : \bH ^2 _d \otimes \mc{H} \rightarrow \hat{\scr{M}}$ is the unitary intertwining $R \otimes I_\mc{H}$ with $\hat{X}$. This can be factored as follows: Since $P_F \oplus P_G \geq P_\wedge$,
\ba \bpm H (L) \\ - H(L) \epm &= & (\mr{e} _F P_F \oplus \mr{e} _G  P_G ) P _\wedge U_H \nn \\
& = & \underbrace{(\mr{e} _F P_F U_F \oplus \mr{e} _G  P_G U_G )}_{=F(L) \oplus G(L)} \underbrace{U_F ^* \oplus U_G ^* U_H}_{=: \Ga } \underbrace{U_H ^* P_\wedge U_H}_{=:Q}. \nn \ea 
In the above $U_F : \bH ^2 _d \otimes \mc{F} \rightarrow \hat{\scr{M} } ^L (F)$ is the unitary intertwining $R\otimes I_\mc{F}$ with $\hat{X} _F$. Observe that $\Ga : \bH ^2 _d \otimes \mc{H} \rightarrow \bH ^2 _d \otimes (\mc{F} \oplus \mc{G} )$ is an isometry and that
\ba \Ga R^\alpha \otimes I_{\mc{H}} & = & (U_F ^* \oplus U_G ^*) \hat{X} ^\alpha U_H \nn \\
& =& (U_F ^* \oplus U_G ^*) (\hat{X}_F \oplus \hat{X} _G) ^\alpha U_H \nn \\
& = & R^\alpha \otimes I_{\mc{F} \oplus \mc{G}} \Ga, \nn \ea so that $\Ga = \Ga (L)$ is an inner left multiplier.  Further observe that,
\ba \left( F (L), G(L) \right) \Ga (L) & = & \left( I_{\bH^2} \otimes I_\cJ , I _{\bH ^2} \otimes I_\cJ \right) \bpm F(L) & 0 \\ 0 & G(L) \epm U_F ^* \oplus U_G ^* U_H \nn \\
& = & \left( I_{\bH^2} \otimes I_\cJ , I _{\bH ^2} \otimes I_\cJ \right) \bpm \mr{e} _F P _F  & 0 \\ 0 & \mr{e} _G P_G  \epm U_H. \nn \ea 
Since $$ U_H : \bH ^2 _d \otimes \mc{H} \rightarrow \hat{\scr{M}} = \bigvee \hat{X}_F ^\alpha \oplus \hat{X}_G ^\alpha \scr{M} ^L _{\bsm H \\ - H \esm}, $$ it follows that the above becomes:
$$  \left( F (L), G(L) \right) \Ga (L) = ( \mr{e} _F , \mr{e} _G ) \underbrace{(P_F \oplus P_G ) U_H } _{\ran{\cdot} \subseteq \scr{M} ^L \bsm H \\ - H \esm} = 0. $$ This proves that the range of the inner $\Ga (L)$ is contained in the kernel of $(F \vee G) (L) = ( F (L) , G(L) )$.
We further claim that 
$$ \bpm H(L) \\ - H(L) \epm = \bpm F(L) & 0 \\ 0 & G(L) \epm \Ga (L) Q = \bpm F(L) & 0 \\ 0 & G(L) \epm \Ga (L). $$ 
Indeed if $\mbf{h} \in \bH ^2 _d \otimes \mc{H}$ belongs to $\mr{Ker} \bsm H(L) \\ - H(L) \esm $ then $U_H \mbf{h} \in \hat{\scr{M}} \ominus \scr{M} ^L \bsm H \\ - H \esm$ (see Remark \ref{ranUF}) and the range space of $\bsm H \\ -H \esm$ is the range of $P_\wedge$ so that $Q \mbf{h} = U_H ^* P _\wedge U_H \mbf{h} = 0$. Furthermore,
$$ F(L) \oplus G(L) \Ga (L) \mbf{h} = \bpm \mr{e} _F P_F & 0 \\ 0 & \mr{e} _G P_G \epm U_H \mbf{h} =0, $$ since $P_F \oplus P_G \hat{\scr{M}} \ominus \scr{M} =0$. If on the other hand, $\mbf{h} \in \mr{Ker} \bsm H \\ - H \esm ^\perp$ then $Q \mbf{h} = \mbf{h}$ so that the claim holds. 
\end{proof}

\subsection*{A special case: injective multipliers}

In this section we assume that $F \in \bH ^\infty _d \otimes \mc{L} ( \mc{F} , \J )$, and $G \in \bH ^\infty _d \otimes \mc{L} (\mc{G}, \J )$ are injective left multipliers.

\begin{cor} \label{isom}
Given $F \in \bH ^\infty _d \otimes \L ( \mc{F} , \J )$, $\scr{M} ^L (F) $ is such that $X := R \otimes I_\J | _{\scr{M} ^L (F)}$ is a row isometry if and only if $\ker{F}$ is $R \otimes I_\mc{F}-$reducing, and,
$$ \ker{F(L)} = \bH ^2 _d \otimes \mc{F} '; \quad \mc{F} \ominus \mc{F} ' = \{ f \in \mc{F} | \  1 \otimes f \in \ker{F(L)} \}. $$
\end{cor}

\begin{lemma} \label{isolemma} Suppose that $\scr{M}$ is boundedly contained in $\bH ^2 _d \otimes \J$, $(R \otimes I_\J)-$invariant, and $X := (R \otimes I_\J )| _{\scr{M}}$ is a row contraction. Then $X$ is a row isometry if and only if $\scr{M} = \scr{M} ^L (F)$ for an injective $F(L) \in \bH ^\infty _d \otimes \L (\mc{F} , \J ) $.
\end{lemma}

\begin{proof} 
We have that $X:= (R \otimes I _\J) | _\scr{M}$ is a pure row isometry by Lemma \ref{pure}.  Recall, by the NC de Branges - Beurling Theorem, that if
$$ F(L) := \mr{e} P _\scr{M} U, $$ where $U : \bH ^2 _d \otimes \mc{F} \rightarrow \hat{\scr{M}}$ is the unitary which intertwines $R \otimes I _\mc{F}$ with the minimal row isometric dilation, $\hat{X}$, of $X$, then $\scr{M} = \scr{M} ^L (F)$. Since we assume that $X$ is its own minimal row isometric dilation, it follows that
$$ F(L) = \mr{e} U, $$ which is injective since $\mr{e}$ is injective and $U$ is unitary.  Conversely, if $F(L)$ is injective, and $X := (R \otimes I_\J) | _\scr{M}$, then for any $\mbf{f}  \in \bH ^2 _d \otimes \mc{F} \otimes \C ^d$, 
\ba \| X (F(L) \otimes I_d) \mbf{f} \| ^2 _\scr{M} & =&  \| F(L) (R \otimes I_\mc{F}) \mbf{f} \| 
\nn \\
& = & \| P _{\ker{F(L)}} ^\perp (R \otimes I_\mc{F} ) \mbf{f} \| ^2 _{\bH ^2 _d\otimes \mc{F}} \nn \\
& = & \| (R \otimes I_\mc{F}) \mbf{f} \| ^2 _{\bH ^2 _d \otimes \mc{F}} \nn \\
& = & \| \mbf{f} \| ^2 _{\bH ^2 _d \otimes \mc{F} \otimes \C ^d}, \nn \ea so that $X$ is a row isometry.
\end{proof}

\begin{proof}{ (of Corollary \ref{isom})}
Let $\scr{M} = \scr{M} ^L (F)$. By the previous lemma we also have that $\scr{M} = \scr{M} ^L (G)$ where $G(L) \in \bH ^\infty _d \otimes \L (\mc{G} , \J )$ is injective. Since the CPNC kernels of $\scr{M} ^L (G) = \scr{M} = \scr{M} ^L (F)$ are the same, it follows that 
$$ F(L) F(L) ^* = G(L) G(L) ^*, $$ so that by Douglas Factorization and commutant lifting, there are contractive operator-valued multipliers, $C (L) \in \bH ^\infty _d \otimes \L (\mc{F} , \mc{G} )$, and $D (L) \in \bH ^\infty _d \otimes \L (\mc{G} , \mc{F} )$ so that $F(L) = G(L) C(L)$ and $G(L) = F(L) D(L)$. 
In particular, 
$$ G(L) = G(L) C(L) D(L), \quad \mbox{and} \quad F(L) = F(L) D(L) C(L). $$ Since $G(L)$ is injective,
$$ C(L) D(L) = I_{\bH ^2 _d \otimes \mc{G}}, $$ and 
$$ P_{\ker{F}} ^\perp = P_{\ker{F} } ^\perp D(L) C(L). $$ Note that $\ker{C(L)} = \ker{F(L)}$ since $G(L)$ is injective. 

We claim that 
$$ D(L) = I_{\bH ^2 _d } \otimes D; \quad \quad D \in \mc{L} ( \mc{G} , \mc{F} ), $$ where $D$ is a fixed isometry. First, if $D(L)$ was not an isometry, then since it is contractive, there would be some $g \in \bH ^2 _d \otimes \mc{G}$ so that $\| D (L) g \| < \| g \| $. But then, 
$$ \| g \| = \| C (L) D(L) g \| \leq \| C(L) \| \| D(L) g \|  < \| C (L) \| \| g \|, $$ which is not possible since $C(L)$ is contractive. Moreover, repeating this same argument for $Z \in \B ^d _\N$ shows that the operator-valued NC function $D(Z)$ is an isometry in $\C ^{n\times n} \otimes \scr{L} (\mc{G} , \mc{F} )$ for any $Z \in \B ^d _n$. Since $C(Z)$ is a contractive left inverse of the isometry $D(Z)$, it follows that $C(Z) = D(Z) ^*$. (This is a consequence of a general fact: A contraction $T$ is an extension of a partial isometry, $V$, in the sense that $T V^* V = V$, if and only if $T^*$ is a contractive extension of $V^*$, see \cite[Lemma 2.3]{JM}. In our case, $C(Z)$ is a contractive extension of the co-isometry $D(Z)^*$ so that $C(Z) ^* $ must be a contractive extension of $D(Z)$, which is an isometry and hence has no non-trivial contractive extensions. That is, $C(Z)^* = D(Z)$, and $C(L) ^* = D(L)$.) 

Further note that $D (Z)$ can be viewed as an operator-valued NC function, \emph{i.e.} it is graded and preserves direct sums and joint similarities, if one conjugates the point evaluations by certain unitary permutation matrices \cite[pp. 65--66]{KVV-rational2}, \cite[p.38]{PV-realize}. However, by the difference-differential calculus for NC functions, this would imply that 
\ba D \bpm Z & X \\ & W \epm  & = & \bpm D(Z) & \Delta D (Z,W) [X]  \\ & D(W) \epm \nn \\
 =  C \bpm Z & X \\ & W \epm ^* & = &  \bpm C(Z) & \Delta C (Z,W) [X]  \\ & C(W) \epm ^* \nn \\
& = & \bpm C(Z) ^* & 0 \\ \Delta C (Z,W) [X] ^* & C(W) ^* \epm. \nn \ea It follows that $\Delta D (Z,W) [X ] \equiv 0$ for any $Z,W, X$, and it follows from Taylor-Taylor series expansions that $D (Z) = I_n \otimes D$, $D := D(0) \in \scr{L} (\mc{G} , \mc{F} )$ for $Z \in \B ^d _n$ is constant-valued \cite[Chapter 7]{KVV}. Moreover, by previous calculation, $D = D(0) \in \scr{L} (\mc{G} , \mc{F} )$ is an isometry.

Now we claim that $\ker{F} ^\perp = \ran{I \otimes D }$. First consider $h= (I \otimes D) y$. Applying the same argument, as above, we see that $C(L)^* = (I \otimes D)^*$. Since $\ker{F} = \ker{C(L)}$, we immediately conclude that $\ker{F} ^\perp = \ran{I \otimes D }$. In conclusion, 
$$ \ker{F (L) } = \bH ^2 _d \otimes \ran{D}, $$ and this completes the proof.

\end{proof}

\begin{thm} \label{injfactor}
If $F,G$ are injective, then so is $H := F \wedge G$ and
$$ \bpm H(L) \\ - H(L) \epm  =  \bpm F(L) & 0 \\ 0 & G(L) \epm \Ga (L), $$ where 
$$ \Ga (L) : \bH ^2 _d \otimes \mc{H} \rightarrow \bH ^2 _d \otimes (\mc{F} \oplus \mc{G} ), $$ is the inner with $\ran{\Ga (L)} = \ker{F\vee G}$.
\end{thm}
\begin{proof}
If both $F,G$ are injective, it follows that $X_F:= R \otimes I_\J | _{\scr{M} ^L (F)}$ and $X_G$ are row isometries. Moreover, as before $\scr{M} := \scr{M} ^L \bsm H \\ - H \esm = \ran{P_\wedge}$ is $X_F \oplus X_G-$invariant so that 
$$ X:= (X_F \oplus X_G) | _{\scr{M} } = X_{\bsm H \\ - H \esm}, $$ is a row isometry, and it follows that $X_H = (R \otimes I_\mc{J} ) | _{\scr{M} ^L (H)}$ is also a row isometry. By Corollary \ref{isom} and Lemma \ref{isolemma}, there is an injective $H'$ so that $\scr{M} ^L (H ) = \scr{M} ^L (H' )$ and we can assume, without loss in generality that $H =H ' $ is injective. 

By the NC de Branges-Beurling theorem, if $U_H : \bH ^2 _d \otimes \mc{H} \rightarrow \scr{M} ^L \bsm H \\ - H \esm$ is the unitary intertwining $R \otimes I_\mc{H}$ with $X =X_{\bsm H \\ - H \esm}$, then
\ba \bpm H(L) \\ - H(L) \epm & = & \mr{e} _F \oplus \mr{e} _G U_H \nn \\
& = & \underbrace{\mr{e}_F U_F \oplus \mr{e} _G U_G}_{=F(L) \oplus G(L)} \underbrace{(U_F ^* \oplus U_G ^*) U_H}_{=\Ga (L)}. \nn \ea 
As in the proof of Theorem \ref{Hid}, $\Ga (L) : \bH ^2 _d \otimes \H \rightarrow \bH ^2 _d \otimes (\mc{F} \oplus \mc{G})$ is inner. Moreover, as in the proof of Theorem \ref{Hid}, $\Ga (L) = U_F ^* \oplus U_G ^* U_H$. If $f \oplus g$ belongs to $\ker{(F,G)}$, then 
\ba 0 & = & \left( F(L) , G(L) \right) \bpm f \\ g \epm \nn \\
& = & \left( \mr{e} _F , \mr{e} _G \right) \bpm U_F & 0 \\ 0 & U_G \epm \bpm f \\ g \epm, \nn \ea 
and it follows that $U_F f \oplus U_G g \in \ran{U_H}$, where $\ran{U_H} = \scr{M} ^L \bsm H \\ - H \esm$, since $H$ is injective. It follows that $\Ga = \Ga (L)$ is onto $\ker{F \vee G}$.
\end{proof}

\section{Point of view: Hilbert modules}

The goal of this section is to describe the operations $\wedge$ and $\vee$ through the language of Hilbert modules and category theory. The $\vee$ we obtain here is slightly different, but only up-to factoring out the kernel, in a sense. Let us call a pair $(\sM,(X_1,\ldots,X_d))$ a pure right Hilbert $\C \langle z_1,\ldots,z_d\rangle$ module, if $(X_1,\ldots,X_d)$ is a pure row contraction. Consider, the category $\hilbmod$ of all pure right Hilbert $\C \langle z_1,\ldots,z_d\rangle$ modules with injective module maps as morphisms. To simplify notations, we will usually omit the tuple of operators and speak simply of a pure Hilbert module $\sM$.

Let $\cC = (\hilbmod \downarrow \bH^2_d)$ be the slice category over the Fock space, on which the free algebra acts by the right free shifts. Namely, the objects of $\cC$ are pairs $(\sM, e)$, where $\sM$ is a pure Hilbert module and $e$ is an embedding of $\sM$ into $\bH^2_d$ as a right-invariant subspace. The morphisms in $\cC$ are $f \colon (\sM,e) \to (\sM',e')$, where $f \colon \sM \to \sM'$ is an injective module map, such that $e = e' \circ f$.

There are two natural operations on $\cC$, the join and the meet. Given two pure Hilbert modules $(\sM_1,e_1)$ and $\sM_2, e_2)$, we consider first the direct sum $\sM_1 \oplus \sM_2$ with the map $e =\left(\begin{smallmatrix} e_1 & e_2 \end{smallmatrix} \right)$ into $\bH^2_d$. Observe that this is a module map and $\ran{e} = \ran{e_1} + \ran{e_2}$. However, this map is, in general, not injective, so we have $\ker{e}$ that is a submodule of $\sM_1 \oplus \sM_2$. Then we get an exact sequence of pure right Hilbert modules
\[
\xymatrix{0 \ar[r] & \ker{e} \ar[r] & \sM_1 \oplus \sM_2 \ar[r]^{e} & \bH^2_d}.
\]
Therefore, set $$(\sM_1,e_1) \vee (\sM_2,e_2) = (\sM_1 \vee \sM_2, e_1 \vee e_2) = (\ker{e}^{\perp},e|_{\ker{e}^{\perp}}).$$

Now note that since both $e_1$ and $e_2$ are injective, the a vector $(\xi, \eta)$ is in $\ker{e}$ if and only if $e_1 \xi = - e_2 \eta$. Hence we consider $\sN = e_1 \sM_1 \cap e_2 \sM_2$. Again, by the injectivity of $e_1$ and $e_2$, each defines an isometry onto its range endowed with the range norm $\|e_j \xi_j\|_j = \|\xi_j\|_j$. We endow $\sN$ with the following norm, if $\zeta= e_1 \xi = e_2 \eta$, then $\|\zeta\|_{\sN}^2 = \| \xi \|_1^2 + \|\eta\|_2^2$. Note that $\sN$ is closed with respect to this norm and the maps $f_1(h) = \xi$ and $f_2(h) = \eta$ are injective and contractive module maps. Thus the map $f \colon \sN \to \sM_1 \oplus \sM_2$ given by $f(h) = (f_1(h), -f_2(h))^T$ is an isometric embedding onto $\ker{e}$. Hence we set $(\sM_1,e_1) \wedge (\sM_2,e_2) = (\ker{e},e_1 P_{\sM_1}) = (\ker{e},-e_2 P_{\sM_2})$ and we get a decomposition of $\sM_1 \oplus \sM_2$ as an orthogonal direct sum of $\sM_1 \wedge \sM_2$ and $\sM_1 \vee \sM_2$. 

\begin{prop} \label{prop:coproduct}
The operation $\vee$ and $\wedge$ endow $\cC$ with coproducts and products, respectively.
\end{prop}
\begin{proof}
Let $\iota_j \colon \sM_j \to \sM_1 \oplus \sM_2$ be the natural embedding, for $j =1,2$. Consider the maps $\varepsilon_j \colon \sM_j \to \sM_1 \vee \sM_2$ that are defined by $\varepsilon_j = P_{\ker{e}}^{\perp} \iota_j$. Then, for $0 \neq \xi \in \sM_1$, we have $\iota_1(\xi) = (\xi,0)^T$ and since $e_1$ is injective, $e(\xi,0)^T = e_1 \xi \neq 0$. Thus, $P_{\ker{e}}^{\perp} \iota_1(\xi) \neq 0$ and we have that 
\[
e \varepsilon_1(\xi) = =e P_{\ker{e}}^{\perp} \iota_1(\xi) = e_1 \xi.
\]
The same is true for $j=2$. Thus, we get that $\varepsilon_j \colon (\sM_j,e_j) \to (\sM_1,e_1) \vee (\sM_2,e_2)$ are morphisms in $\cC$. Now, if $(\sN,f)$ is another object in $\cC$, with maps $\theta_j \colon (\sM_j,e_j) \to (\sN,f)$, then $f \theta_j = e_j = e \iota_j$. Set $\theta \colon \sM_1 \oplus \sM_2 \to \sN$, $\theta =\left(\begin{smallmatrix} \theta_1 & \theta_2 \end{smallmatrix} \right)$. Thus, $f \theta = e$ and we obtain a map $\widetilde{\theta} = \theta|_{\ker{e}^{\perp}}$ from  $(\sM_1,e_1) \vee (\sM_2,e_2)$ to $(\sN,f)$. Moreover,
\[
f \theta_j = e_j = (e_1 \vee e_2) \varepsilon_j = f \widetilde{\theta} \varepsilon_j.
\]
Since $f$ is injective, we conclude that $\theta_j = \widetilde{\theta} \varepsilon_j$. Now it remains to prove uniqueness. Let $\varphi \colon (\sM_1,e_1) \vee (\sM_2,e_2) \to (\sN,f)$ be another map, such that $\varphi \varepsilon_j = \theta_j$. Then $\varphi \varepsilon_j = \widetilde{\theta} \varepsilon_j$. However, $\left(\begin{smallmatrix} \varepsilon_1 & \varepsilon_2 \end{smallmatrix} \right)$ is surjective and thus $\varphi = \widetilde{\theta}$.

Now consider $(\sM_1,e_1) \wedge (\sM_2,e_2)$. As we have seen, there are two (contractive) homomorphism $f_j \colon (\sM_1,e_1) \wedge (\sM_2,e_2) \to (\sM_j,e_j)$, $j=1,2$. Now given $(\sN,g)$ with two homomorphisms $h_j \colon (\sN,g) \to (\sM_j,e_j) $, $j=1,2$, set $h = \left(\begin{smallmatrix} h_1 & \\ - h_2 \end{smallmatrix} \right)^T$. Then, since $g = e_1 h_1 = e_2 h_2$, we get that $\ran{h} \subset \ker{e}$. Since $M_1 \wedge M-2 = \ker{e}$, to see that this is a morphism in $\cC$, we only need to compute
\[
g = e_1 h_1 = e_1 = e_1 P_{\sM_1} h.
\]
To prove uniqueness, assume that $\chi \colon \sN \to \ker{e}$ is another map, such that $g = e_1 P_{\sM_1} \chi$. Since $e_1 P_{\sM_1}$ is injective on $\ker{e}$, we get that $\chi = h$.
\end{proof}

\begin{lemma}
If $(\sM_1,e_1), (\sM_2,e_2) \in \cC$, then $(\sM_1,e_1) \cong (\sM_2,e_2)$ if and only if $\ran{e_1} = \ran{e_2}$.
\end{lemma}
\begin{proof}
Of course, if $(\sM_1,e_1) \cong (\sM_2,e_2)$, then there exists an isomorphism $f \colon \sM_1 \to \sM_2$, such that $e_2 f = e_1$ and thus $\ran{e_2} = \ran{e_1}$.

Conversely, if $\ran{e_1} = \ran{e_2} = \sN$, then $f_j \colon \sN \to \sM_j$, $j = 1,2$, are contractive bijections. Hence an isomorphism is given by $f_2^{-1} f_1$.
\end{proof}

Therefore, we can construct a skeleton of the category, by simply considering the ranges of the embeddings. By the noncommutative de Branges-Beurling theorem, for every $(\sM,e) \in \cC$, there exists a Hilbert space $\J$ and a contractive row $F(L) \in \bH^{\infty}_d \otimes B(\J,\C)$, such that $e\sM = F(L)( \bH^2_d \otimes \J)$, where, the norm on $\sM$ is $\|\cdot\|_F$. The multiplier $F$ is of course non-unique, however, by Corollary \ref{Multfactor}, there is a natural choice for such a multiplier, namely $F(L) = e P_{\sM} U$. We will call this multiplier the representative of $(\sM,e)$.

\begin{lemma} \label{lem:equiv_reps}
If $F(L) \in \bH^{\infty}_d \otimes B(\J,\C)$ and $G(L) \in \bH^{\infty}_d \otimes B(\K,\C)$ are two representatives of $(\sM,e)$, then there exists an invertible multiplier $O(L) \in \bH^{\infty}_d \otimes B(\J,\K)$, such that $G(L) = F(L) O(L)$.
\end{lemma}
\begin{proof}
Contained in the proof of the lattice theorem.
\end{proof}

Due to the lemma, we can make the following definition:

\begin{defn}
For $(\sM,e)$ we define $\dim{\sM,e} = \dim{\J}$, where $F(L) \in \bH^{\infty}_d \otimes B(\J,\C)$ is a representative of $(\sM,e)$.
\end{defn}

Conversely, given $F(L) \colon \bH ^2_d \otimes \J \to \bH ^2_d$, we can define $\sM_F = \ker{F(L)}^{\perp}$ and $e_F = F(L)|_{\ker{F}^{\perp}}$. Then, $(\sM_F,e_F) \in \cC$. If $\sM \subset \bH ^2_d$ is a closed right-invariant subspace, then $\sM$ is the image of an isometry $F$. Since the right shifts restricted to $\sM$ form a row isometry, we have that a representative of $(\sM,e)$ is $e$ itself. Hence we make the following definition.

\begin{defn}
Let $F(L) \colon \bH ^2_d \otimes \J \to \bH ^2_d$. We say that $F$ is minimal, if $F$ is a representative of $(\sM_F,e_F)$.
\end{defn}

By Lemma \ref{lem:equiv_reps}, if $F$ and $G$ are minimal with the same range, then $G = F O$, for some invertible operator-valued multiplier $O(L)$.

\begin{defn}
Let $F(L) \colon \bH ^2_d \otimes \J \to \bH ^2_d$ and $G(L) \colon \bH ^2_d \otimes \K \to \bH ^2_d$, we set 
\begin{itemize}
\item $F \wedge G$ is a representative of $(\sM_F,e_F) \wedge (\sM_G,e_G)$. 

\item $F \vee G$ is a representative of $(\sM_F,e_F) \vee (\sM_G,e_G)$.
\end{itemize}
\end{defn}

\begin{remark}
The lattice properties follow by standard arguments from Proposition \ref{prop:coproduct}. In particular, it is a general fact of category theory that for three objects $X,Y,Z$ in a category with products and co-products.
\[
X \times (Y \sqcup Z) = (X \times Y) \sqcup (X \times Z).
\]
\end{remark}

\begin{thm}
Let $\cS$ be the collection of ranges of minimal multipliers. Then, $\cS$ is a bounded general lattice.
\end{thm}

\bibliography{dB}

\end{document}